\documentclass[12pt,reqno]{amsart}
\usepackage{amssymb}
\usepackage{mathrsfs}
\usepackage{cite}
\usepackage{tikz-cd}
\usepackage{MnSymbol}
\usepackage{a4wide}
\usepackage{mathtools}
\makeatletter
\let\@wraptoccontribs\wraptoccontribs
\makeatother

\DeclareMathOperator\C{\mathbb C}

\newtheorem{theorem}{Theorem}[section]

\newtheorem{lemma}[theorem]{Lemma}
\newtheorem{cor}[theorem]{Corollary}
\newtheorem{prop}[theorem]{Proposition}
\theoremstyle{definition}
\newtheorem{definition}[theorem]{Definition}
\newtheorem{example}[theorem]{Example}

\theoremstyle{remark}
\newtheorem{remark}[theorem]{Remark}

\numberwithin{equation}{section}

\newcommand{\dontprint}[1]\relax

\newcommand{\sspan}{\operatorname{span}}

\newcommand{\we}{\wedge}

\renewcommand{\P}{{\mathbb P}}

\newcommand{\wt}{\widetilde}
\newcommand{\ot}{\otimes}
\newcommand{\Hom}{\operatorname{Hom}}

\newcommand{\VV}{{\mathcal V}}

\newcommand{\CC}{{\mathcal C}}

\newcommand{\KK}{{\mathcal K}}

\newcommand{\OO}{{\mathcal O}}

\newcommand{\UU}{{\mathcal U}}
\newcommand{\WW}{{\mathcal W}}

\newcommand{\sub}{\subset}

\newcommand{\ov}{\overline}
\newcommand{\im}{\operatorname{im}}

\newcommand{\la}{\lambda}
\renewcommand{\a}{\alpha}

\newcommand{\id}{\operatorname{id}}

\newcommand{\tot}{\operatorname{tot}}

\renewcommand{\th}{\theta}
\newcommand{\ga}{\gamma}

\newcommand{\rk}{{\operatorname{rk}}}
\newcommand{\codim}{{\operatorname{codim}}}
\newcommand{\sym}{{\operatorname{sym}}}
\renewcommand{\k}{{\mathbf{k}}}
\newcommand{\Sing}{{\operatorname{Sing}}}

\title{Schmidt rank and singularities}

\author{David Kazhdan}
\author{Amichai Lampert}
\author{Alexander Polishchuk}
\thanks{A.L. is supported by the National Science Foundation
under Grant No. DMS-1926686 and by the Israel Science Foundation under Grant No. 2112/20}
\thanks{A.P. is partially supported by the NSF grant DMS-2001224, 
and within the framework of the HSE University Basic Research Program and by the Russian Academic Excellence Project `5-100'.}

\address{Einstein Institute of Mathematics,
The Hebrew University of Jerusalem,
Jerusalem 91904, Israel}
\email{kazhdan@math.huji.ac.il}

\address{Einstein Institute of Mathematics,
The Hebrew University of Jerusalem,
Jerusalem, Israel; and Institute for Advanced Study, Princeton, NJ, USA}
\email{amichai.lam@gmail.com}

\address{
    Department of Mathematics, 
    University of Oregon, 
    Eugene, OR 97403, USA; National Research University Higher School of Economics; and Korea Institute for 
    Advanced Study 
  }
  \email{apolish@uoregon.edu}

\begin{document}

\begin{abstract}
We revisit Schmidt's theorem connecting the Schmidt rank of a tensor with the codimension of a certain
variety and adapt the proof to the case of arbitrary characteristic. We also find a sharper result of this kind
for homogeneous polynomials, assuming the characteristic does not divide the degree. We then use this
to relate the Schmidt rank of a homogeneous polynomial (resp., a collection of homogeneous polynomials of the same degree) with the codimension of the singular locus of the corresponding hypersurface (resp., intersection of hypersurfaces). This gives an effective version of Ananyan-Hochster's theorem \cite[Theorem A]{AH}.
\end{abstract}

\maketitle

\section{Introduction}

Let $\k$ be a field (of any characteristic) and
$$P:V_1\times V_2\times \ldots\times V_d \to \k$$
a polylinear form, where $V_i$ are finite dimensional vector spaces over $\k$. Equivalently, we view $P$ as a tensor in $V_1^*\ot\ldots\ot V_d^*$.

\begin{definition}\label{S-rank-def}
(i) We say that $P\neq 0$ has {\it Schmidt rank} $1$ if there exists a partition $[1,d]=I\sqcup J$ into two nonempty
parts and polylinear forms $P_I(v_{i_1},\ldots,v_{i_r})$, $P_J(v_{j_1},\ldots,v_{j_s})$, where
$v_a\in V_a$, $I=\{i_1<\ldots<i_r\}$, $J=\{j_1<\ldots<j_r\}$, such that
$P=P_I\cdot P_J$. In general the {\it Schmidt rank} of $P$, denoted as $\rk^S(P)$,
is the smallest number $r$ such that $P=\sum_{i=1}^r P_i$
with $P_i$ of Schmidt rank $1$. 
For a collection of tensors $\ov{P}=(P_1,\ldots,P_s)$ we define the {\it Schmidt rank} $\rk^S(\ov{P})$ as
the minimum of Schmidt ranks of nontrivial linear combinations of $(P_i)$. 

\noindent
(ii) Given a collection of nonempty subsets $I_1,\ldots,I_r\sub [1,d]$ and a collection $(P_{I_1},\ldots,P_{I_r})$,
where $P_{I_i}$ is a polylinear form on $\prod_{a\in I_i} V_a$,
we denote by $(P_{I_1},\ldots,P_{I_r})\sub V_1^*\ot\ldots\ot V_d^*$ and call this the {\it tensor ideal generated by
$P_{I_1},\ldots,P_{I_r}$}, the subspace of polylinear forms of the form
$$P=\sum_{i=1}^r P_{I_i}\cdot Q_{J_i},$$
for some polylinear forms $Q_{J_i}$ on $\prod_{b\in J_i} V_b$ , where $J_i=[1,d]\setminus I_i$. 
\end{definition}

The Schmidt rank of a tensor, along with a set of related notions, such as {\it slice rank}, {\it $G$-rank}, {\it analytic rank},
as well as the version of Schmidt rank for homogeneous polynomials also known as {\it strength} (see below), 
has been a subject of study in many
recent works (see \cite{AKZ}, \cite{BBOV1}, \cite{BBOV2}, \cite{Derksen}, \cite{KZ} and references therein). One of the goals of this paper is to establish a precise relation 
(in the case of an algebraically closed base field $\k$) between this notion and the codimension
of the singular locus of the corresponding hypersurface, thus giving an effective version of Ananyan-Hochster's theorem \cite[Theorem A]{AH}.

Let us define the subvariety $Z_P=Z_P^{V_1}\sub V_2\times \ldots\times V_d$ as the set of $(v_2,\ldots,v_d)$ such that
$P(v_1,v_2,\ldots,v_d)=0$ for all $v_1\in V_1$. 
Following Schmidt, let us set
$$g(P):=\codim_{V_2\times\ldots \times V_d} Z_P.$$
In \cite{KMZ} (where the authors consider the case $d=3$), this number is called {\it geometric rank} of $P$. Using \cite[Thm.\ 3.2]{KMZ}, one can see that it does not depend on the ordering of
the variables $v_1,\ldots,v_d$.

It is easy to see that one has
\begin{equation}\label{trivial-ineq}
g(P)\le \rk^S(P)
\end{equation}
(see Lemma \ref{trivial-ineq-lem}(i) below, or \cite[Thm.\ 1]{KMZ}).

Similarly, for a collection $\ov{P}=(P_1,\ldots,P_s)$, we define $Z_{\ov{P}}\sub V_2\times\ldots\times V_d$ 
by the condition on $(v_2,\ldots,v_d)$ that the corresponding map
$$V_1\to \k^s: v_1\mapsto (P_i(v_1,v_2,\ldots,v_d))_{1\le i\le s}$$
has rank $<s$, and we set
$g(\ov{P}):=\codim_{V_2\times\ldots \times V_d} Z_{\ov{P}}$.

The proof of the following theorem follows closely the proof of a similar result in the case where $\k=\C$ and $P$ is symmetric, 
given in \cite{Schmidt}. We modified the proof so that it would work in arbitrary characteristic and also streamlined some parts of the original argument.
The fact that the original proof can be adapted to arbitrary characteristic was also pointed out in \cite[Sec.\ 4]{Schmidt}.

\begin{theorem}\label{main-thm}
(i) Let $g'(P)$ denote the codimension in $V_2\times\ldots\times V_d$ of the Zariski closure of the set of $\k$-points in $Z_P$
(so $g(P)\le g'(P)$ and $g(P)=g'(P)$ if $\k$ is algebraically closed).
Then one has
$$\rk^S(P)\le C_d\cdot g'(P),$$
where $C_d=\max(2+\th_{d-2},2^{d-2}-1)$, and 
$\th_n$ is the number of ordered collections of disjoint nonempty
subsets $I_1\sqcup\ldots\sqcup I_p\subsetneq [1,n]$ (with $p\ge 1$).
In particular, we have $C_3=2$, $C_4=4$ and $C_5=14$.

\noindent
(ii) Assume $\k$ is algebraically closed. Then
for a collection $\ov{P}=(P_1,\ldots,P_s)$, one has
$$\rk^S(\ov{P})\le C_d\cdot (g(\ov{P})+s-1).$$
\end{theorem}

In the appendix we prove another version of theorem \ref{main-thm} with better bounds for $d\ge 6.$
Even though Schmidt applied the above result to symmetric tensors $P$ corresponding to homogeneous polynomials,
we observe that in the symmetric case it is natural to modify the relevant variety $Z_P$, and that this leads to
much better estimates on the rank.

Let $f$ be a homogeneous polynomial of degree $d$ on a finite-dimensional $\k$-vector space $V$.
The {\it Schmidt rank} of $f$, denoted as $\rk^S(f)$, is the minimal number $r$ such that $f=\sum_{i=1}^r g_ih_i$, where $g_i$ and $h_i$ are homogeneous polynomials of positive degrees. 
Note that if $\rk^S(f)=r$ then in terminology of \cite{AH}, $f$ has {\it strength} $r-1$.
For a collection $\ov{f}=(f_1,\ldots,f_s)$, the Schmidt rank
$\rk^S(\ov{f})$ is defined as the minimum of Schmidt ranks of nontrivial linear combinations of $f_i$.

Let $H_f(x)(\cdot,\cdot)$ 
denote the Hessian form of $f$ given by the second derivatives of $f$. It is a symmetric bilinear form on $V$
depending polynomially on a point $x\in V$. The symmetric analog of the variety $Z_P$ is the subvariety
$Z^{\sym}_f\sub V\times V$ given as
$$Z^{\sym}_f:=\{(v,x)\in V\times V \ |\ v\in \ker H_f(x)\}.$$
Let us set 
$$g_{\sym}(f):=\codim_{V\times V}(Z^{\sym}_f).$$
The symmetric analog of \eqref{trivial-ineq} is the inequality
\begin{equation}\label{sym-trivial-ineq}
g_{\sym}(f)\le 4\rk^S(f)
\end{equation}
(see Lemma \ref{trivial-ineq-lem}(ii)).

Similarly for a collection $\ov{f}=(f_1,\ldots,f_s)$ of homogeneous polynomials of degree $d$, we define
the subvariety $Z^{\sym}_{\ov{f}}\sub V\times V$ as the set of $(v,x)$ such that the map 
$$V\to\k^s: v'\mapsto (H_{f_i}(x)(v',v))_{1\le i\le s}$$
has rank $<s$. We denote by $g_{\sym}(\ov{f})$ the codimension of $Z^{\sym}_{\ov{f}}$ in $V\times V$.

\begin{theorem}\label{sym-main-thm}
(i) Assume that $d\ge 3$ and that the characteristic of $\k$ does not divide $(d-1)d$.
Let $g'_{\sym}(f)$ denote the codimension in $V\times V$ of the Zariski closure of the set of $\k$-points in $Z^{\sym}_f$.
Then one has
$$\rk^S(f)\le 2^{d-3}\cdot g'_{\sym}(f).$$ 

\noindent
(ii) With the same assumptions as in (i), assume also that $\k$ is algebraically closed.
Then
$$\rk^S(\ov{f})\le 2^{d-3}\cdot (g_{\sym}(\ov{f})+s-1).$$
\end{theorem}

For $\k$ algebraically closed, we prove another version of theorem \ref{sym-main-thm} in the appendix with better bounds for $d\ge 6.$ The invariant $g_{\sym}(f)$ can be viewed as an invariant measuring singularities of the polar map
$x\mapsto (\partial_i f(x))_{1\le i\le \dim V}$ of $f$ (see Sec.\ \ref{sing-sec}).
We also prove that $g_{\sym}(f)$ is related to the codimension of the singular locus of the hypersurface $f=0$.
Namely, let us set
$$c(f):=\codim_V \Sing(f=0).$$
Assuming that $\mathrm{char}(\k)$ does not divide $2(d-1)$, we prove that
\begin{equation}\label{cf-ineq}
\begin{array}{c}
c(f)\le g_{\sym}(f)\le (d+1)\cdot c(f), \ \text{ for } d \ \text{ even},  \nonumber\\
c(f)\le g_{\sym}(f)\le d\cdot c(f), \ \text{ for } d \ \text{ odd}
\end{array}
\end{equation}
(see Proposition \ref{cf-ineq-prop}).

More generally, for a collection $\ov{f}=(f_1,\ldots,f_s)$, let us set
$$c(\ov{f}):=\codim_V \Sing(V(\ov{f})),$$
where $V(\ov{f})\sub V$ is the subscheme defined by the ideal $(f_1,\ldots,f_s)$.
We also consider the related invariant
$$c'(\ov{f}):=\codim_V S(\ov{f}),$$
where $S(\ov{f})\sub V$ is the locus where the Jacobi matrix of $(f_1,\ldots,f_s)$ has rank $<s$.
It is easy to see that
$$c'(\ov{f})\le c(\ov{f})\le c'(\ov{f})+s.$$

Here is our main result concerning the relation between the Schmidt rank and the codimension of the singular locus.
It can be viewed as a more precise version of the corresponding result in \cite{KZ} in the case of
an algebraically closed field of sufficiently large (or zero) characteristic, as well as an effective version of a result of Ananyan and Hochster (see \cite[Theorem A(a)]{AH}),
playing a central role in their proof of Stillman's conjecture.

\begin{theorem}\label{cf-thm}
Assume that $\mathrm{char}(\k)$ does not divide $d$. Let $c_\k(f)$ be the codimension in $V$ of the Zariski closure of the $\k$-points of $\Sing(f=0).$

\noindent
(i) We have
$$\frac{c(f)}{2} \le \rk^S(f)\le (d-1)\cdot c_\k(f).$$

\noindent
(ii) Assume $\k$ is algebraically closed. Then for a collection $\ov{f}=(f_1,\ldots,f_s),$ we have
$$\rk^S(\ov{f})\le (d-1)\cdot (c'(\ov{f})+s-1).$$
%
\end{theorem}

Combining Theorem \ref{cf-thm}(i) with \cite[Theorem A(c)]{AH} we get the following result.

\begin{cor}\label{AH-cor}
Assume $\k$ is algebraically closed and $\mathrm{char}(\k)$ does not divide $d!$. 
For $i=2,\ldots,d$, let $W_i\sub \k[V]_i$ be a subspace of forms of degree $i$. Set $W=\bigoplus_i W_i$, $w=\dim W$.
Assume that for some $m\ge 1$, one has
$$\rk^S(W_i)\ge (i-1)(m+2)+3(w-1), \text{ for } i=3,\ldots,d,$$
$$\rk^S(W_2)-1\ge \lceil \frac{m+1}{2}\rceil+3(w-1).$$
Then every sequence of linearly independent homogeneous forms in $W$ is regular and the corresponding complete intersection subscheme in $V$ satisfies Serre
condition $R_m$.
\end{cor}

Note that without any assumptions on the characteristic on $\k$ we are able to estimate in terms of $c(f)$
the rank of $H_f(x)(u,v)$ viewed as a polynomial in $(u,v,x)\in V\times V\times V$
(see Remark \ref{char-free-rem}).

For a homogeneous polynomial $f(x)$ of degree $d$ on $V$ and a vector $v\in V$, we denote by
$\partial_v f(x)$, the derivative of $f$ in the direction $v$. Our next result concerns $\partial_v f$ for generic $v$.

\begin{theorem}\label{derivative-thm}
Let $f$ be a homogeneous polynomial of degree $d\ge 3$.
Assume that $\k$ is algebraically closed of characteristic not dividing $(d-1)d$.

\noindent
(i) For generic $v\in V$, one has $\rk^S(\partial_v f)\ge 2^{2-d}\cdot \rk^S(f)$.

\noindent
(ii) For $s\le 2^{2-d}\cdot \rk^S(f)+\frac{1}{2}$ (resp., $s\le 2^{2-d}\cdot \rk^S(f)-\frac{1}{2}$), and for generic $v_1,\ldots,v_s\in V$, the derivatives $(\partial_{v_1}f,\ldots,\partial_{v_s}f)$ define a (resp., normal)
complete intersection of codimension $s$ in $V$.
\end{theorem}

In the appendix we prove another version of theorem \ref{derivative-thm} with better bounds for $d\ge 6.$ In Section \ref{sing-polar-sec} we will also discuss the relation of the invariant $g_{\sym}(f)$ with
the polar map of $f$ and with the Gauss map of the corresponding projective hypersurface.

\section{Schmidt rank of polylinear forms}

\subsection{Elementary observations}

First, let us prove \eqref{trivial-ineq} and its symmetric version \eqref{sym-trivial-ineq}.
We denote by $\k[V]$ the space of polynomial functions on a vector space $V$ and by $\k[V]_d\sub \k[V]$ the subspace
of homogeneous polynomials of degree $d$.

\begin{lemma}\label{trivial-ineq-lem}
(i) For $P\in V_1^*\ot\ldots\ot V_d^*$ one has $g(P)\le \rk^S(P)$.

\noindent
(ii) For $f\in \k[V]_d$ one has $g_{\sym}(f)\le 4\rk^S(f)$.
\end{lemma}

\begin{proof} (i) If $r=\rk^S(P)$ then there exists a decomposition
$$P=\sum_{i=1}^r P_{I_i}\cdot Q_{J_i}$$
as in Definition \ref{S-rank-def}. Swapping some $I_i$ with $J_i$ if necessary, we can assume that
$1\in I_i$ for all $i$. Then the intersection of $r$ hypersurfaces $Q_{J_i}=0$ in $V_2\times\ldots\times V_d$
is contained in $Z_P$ and has codimension $\le r$.

\noindent
(ii) If we have a decomposition $f=\sum_{i=1}^r g_ih_i$ then over the subvariety $Y=V(g_1,\ldots,g_r,h_1,\ldots,h_r)\sub V$ the symmetric form $H_f(x)$ has rank
$\le 2r$: the subspace cut out by $dg_1|_x,\ldots,dg_r|_x,dh_1|_x,\ldots,dh_r|_x$ is contained in its kernel. Since $\codim_V Y\le 2r$, the preimage of $Y$ in $Z^{\sym}_f$
has codimension $\le 4r$ in $V\times V$.
\end{proof}

For a subset of indices $I=\{i_1<\ldots<i_s\}\sub [1,d]$, let us set 
$$V_I:=V_{i_1}\ot\ldots\ot V_{i_s}.$$
We have the following simple observation.

\begin{lemma}\label{elem-lem}
Let $V'_1\sub V_1$ be a subspace of codimension $c$, 
and let $(\ell_1,\ldots,\ell_r)$ be a basis of the orthogonal to $V'_1$ in $V^*_1$.
Suppose we have tensors 
$$P_{I_s}\in V_1^*\ot V_{I_s}^*, \ \ Q_{J_t}\in V_{J_t}^*$$ 
for some subsets $I_1,\ldots,I_r,J_1,\ldots,J_p\sub [2,\ldots,d]$, such that
$P|_{V'_1\times V_2\times\ldots V_d}$ belongs to
the tensor ideal 
$$(P_{I_s}|_{V'_1\ot V_{I_s}},Q_{J_t} \ |\ s=1,\ldots,r; t=1,\ldots,p).$$
Then $P$ belongs to the tensor ideal 
$$((\ell_i \ |\ i=1,\ldots,c),(P_{I_s},Q_{J_t} \ |\ s=1,\ldots,r; t=1,\ldots,p)).$$
In particular,
$$\rk^S(P)\le \rk^S(P|_{V'_1\times V_2\times\ldots\times V_d})+c.$$
\end{lemma}

\begin{proof}
This follows immediately from the fact that the tensor ideal $(\ell_i \ |\ i=1,\ldots,c)$ is exactly the kernel of the
restriction map 
$$(V_1\otimes V_2\otimes \ldots\otimes V_d)^*\to (V'_1\otimes V_2\otimes \ldots\otimes V_d)^*.$$
\end{proof}

\subsection{Determinantal construction}

Let $f:\VV_1\to \VV_2$ be a morphism of vector bundles on a scheme $X$.
For every $r\ge 0$, we have a natural morphism
$$\kappa_r:{\bigwedge}^r\VV_2^\vee\ot {\bigwedge}^{r+1}\VV_1\to \VV_1:
(\phi_1\we\ldots\we\phi_r)\ot \a\mapsto \iota_{f^\vee\phi_1}\ldots\iota_{f^\vee\phi_r}\a,$$
where for a section $\psi$ of $\VV^\vee$ we denote by $\iota_\psi:\bigwedge^i\VV\to \bigwedge^{i-1}\VV$
the corresponding contraction operator.

\begin{lemma}\label{det-lem}
(i) Assume that ${\bigwedge}^{r+1}f=0$. Then the image of $\kappa_r$ is contained in $\ker(f)$.

\noindent
(ii) Assume in addition that $\VV_1$ and $\VV_2$ are trivial vector bundles and that for some point $x\in X$,
the rank of $f(x):\VV_1|_x\to \VV_2|_x$ is equal to $r$. Let $n=\rk \VV_1$.
Then there exist $n-r$ global sections $s_1,\ldots,s_{n-r}$ of $\VV_1$ such that $f(s_i)=0$ for all $i$, and
$s_1(x),\ldots,s_{n-r}(x)$ is a basis of $\ker f(x)$.
\end{lemma}

\begin{proof}
(i) This is equivalent to the statement that $\iota_{f^\vee\phi_{r+1}}\kappa_r(\a)=0$ for any local section $\phi_{r+1}$
of $\VV_2^{\vee}$. But $\iota_{f^\vee\phi_1}\ldots\iota_{f^\vee\phi_r}\iota_{f^\vee\phi_{r+1}}=0$ since
${\bigwedge}^{r+1}f^\vee=0$.

\noindent
(ii) Since $\VV_1$ and $\VV_2$ are trivial, we can choose splittings $\VV_1=\KK\oplus \WW_1$,
$\VV_2=\CC\oplus \WW_2$ into trivial subbundles, 
such that $\KK|_x=\ker f(x)$, $\WW_2|_x=\im f(x)$, and $f(x):\WW_1|_x\to \WW_2|_x$
is an isomorphism. Let us consider the composed map
$$s:{\bigwedge}^r\WW_2^\vee\ot ({\bigwedge}^r\WW_1\ot \KK)\to 
{\bigwedge}^r\VV_2^\vee\ot {\bigwedge}^{r+1}\VV_1\xrightarrow{\kappa_r} \VV_1.$$
Then $f\circ s=0$ and the image of $s(x)$ is exactly $\ker f(x)$.
Choosing a trivialization of the target of $s$, we can write $s$ as a collection of global sections of $\VV_1$, which has the
required properties.
\end{proof}

\subsection{Higher derivatives}

Let $V$ be a finite dimensional vector space, and let $\k[V]$ denote the ring of polynomial functions on $V$.

For each $f\in \k[V]$, each $n\ge 1$ and $v_0\in V$, 
we define the homogeneous form of degree $n$ on $V$, $f^{(n)}_{v_0}(v)$ as the
$n$th graded component of $f(v+v_0)\in \k[V]$ (viewed as a function of $v$, for fixed $v_0$)
with respect to the degree grading on $\k[V]$, so that we have 
(finite) Taylor's decomposition
$$f(v+v_0)=\sum_{n\ge 0} f^{(n)}_{v_0}(v).$$
We refer to $f^{(n)}_{v_0}$ as the $n$th derivative of $f$ at $v_0$.

\begin{lemma}\label{deriv-ideal-lem}
Let $X\sub V$ be an irreducible closed subvariety of codimension $c$, $v_0\in X$ a smooth $\k$-point. 
Let $g_1,\ldots,g_c$ be a set of elements in the ideal $I_X$ of $X$, with linearly independent differentials at $v_0$.
Then for any $f\in I_X$ and any $n\ge 1$, the form $f^{(n)}_{v_0}\in \k[V]$ belongs to the ideal in $\k[V]$ generated by
$((g_i)^{(j)}_{v_0})_{i=1,\ldots,c;1\le j\le n}$.
\end{lemma}

\begin{proof} Without loss of generality we can assume that $v_0=0$. Set $A=\k[V]$, and let $\hat{A}$ denote the completion with respect
to the ideal of the origin (the ring of formal power series). Then the key point is that 
$I_X\cdot \hat{A}$ is generated by $g_1,\ldots,g_c$. Indeed, this follows from the fact that
the local homomorphism of local regular $\k$-algebras $A_{\frak{m}}/(g_1,\ldots,g_c)\to \OO_{X,v_0}$ (where $\frak{m}$
is the maximal ideal of $v_0$ in $A$) induces an
isomorphism on tangent spaces, so it induces an isomorphism of completions. Note that higher derivatives make sense
for elements of $\hat{A}$ (as components in $A_n=\k[V]_n$), so the assertion follows once we express 
any element of $I_X$ in the form $\sum_i g_i h_i$ for some $h_i\in \hat{A}$.
\end{proof}

We also need to work with certain polylinear forms of mixed derivatives.
Assume that we have a decomposition $V=V_1\oplus\ldots\oplus V_n$. 
Then we have the induced direct sum decomposition
$$\k[V]_m=\bigoplus_{m_1+\ldots+m_n=m} \k[V_1]_{m_1}\ot\ldots\ot \k[V_n]_{m_n}.$$ 

Now for $f\in \k[V]_m$ with $m\le n$, and a subset of indices $1\le i_1<\ldots<i_m\le n$,
we denote by $f^{(V_{i_1},\ldots,V_{i_m})}$ the component of $f$ in $\k[V_{i_1}]_1\ot\ldots\ot \k[V_{i_m}]_1$.
In particular, when we apply this to the $m$th derivative of $f$ at $v_0$, we get a polylinear form
\begin{equation}\label{mixed-der-eq}
f^{(V_{i_1},\ldots,V_{i_m})}_{v_0}:=(f^{(m)}_{v_0})^{(V_{i_1},\ldots,V_{i_m})}\in V_{i_1}^*\ot\ldots\ot V_{i_m}^*,
\end{equation}
which we call the {\it $(V_{i_1},\ldots,V_{i_m})$-mixed derivative of $f$ at $v_0$}.

\begin{lemma}\label{mixed-deriv-ideal-lem}
In the situation of Lemma \ref{deriv-ideal-lem}, assume in addition that $V=V_1\oplus\ldots\oplus V_n$.
Then for any $f\in \k[V]$, and any collection of indices
$I=\{i_1<\ldots<i_m\}\sub [1,n]$, the polylinear form $f^{(V_{i_1},\ldots,V_{i_m})}_{v_0}$ belongs
to the tensor ideal generated by $(g_i)^{(V_{j_1},\ldots,V_{j_s})}_{v_0}$, for $i=1,\ldots,c$ and
$J=\{j_1<\ldots<j_s\}\sub I$, $J\neq\emptyset$.
\end{lemma}

\begin{proof} This follows easily from Lemma \ref{deriv-ideal-lem}.
\end{proof}

\subsection{Dimension count}

Let us change the notation to
$$P:U\times V\times W_1\times\ldots\times W_{d-2}\to \k.$$
We will denote $W=W_1\times\ldots\times W_{d-2}$,
and consider the variety $Z_P\sub V\times W$ of all $(v,w)$ such that $P(u,v,w)=0$ for all $u\in U$.

Let $Z$ be an irreducible component of the Zariski closure of the set of $\k$-points
$Z_P(\k)$ (with reduced scheme structure), such that $\codim_{V\times W} Z=g'(P)$, and let 
$Z_W\sub W$ denote the closure of the image of $Z$ under the projection $\pi_W:V\times W\to W$ (also with reduced scheme
structure). Then $\k$-points are dense in $Z_W$.

We can think of $P$ as a linear map from $U\ot V$ to the space of polynomial functions on $W$, hence, it
gives a morphism of trivial vector bundles over $W$,
\begin{equation}\label{P_W-morphism}
P_W:V\ot \OO_W\to U^*\ot \OO_W,
\end{equation}
and for $w\in Z_W$, $\pi_W^{-1}(w)\cap Z_P$ can be identified with $\ker(P_W(w))$. 

Let $\UU\sub Z_W$ denote the nonempty open subset where $P_W$ has maximal rank that we denote by $r$.
Then over $\UU$ the cokernel of $P_W$ is locally free over $Z_W$, hence, the kernel of $P_W$ is a subbundle
$\KK\sub V\ot \OO$. Denoting by $\tot_{\UU}(\KK)$ the total space of the bundle $\KK$ over $\UU$, we have
$$\tot_{\UU}(\KK)=\pi_W^{-1}(\UU)\cap Z_P\sub V\times W.$$ 
Note that $\k$-points are dense in $\tot_{\UU}(\KK)=\pi_W^{-1}(\UU)\cap Z_P$, so
$\pi_W^{-1}(\UU)\cap Z$ is an irreducible component in $\pi_W^{-1}(\UU)\cap Z_P$.
Since $\tot_{\UU}(\KK)$ is irreducible, we get
$$\pi_W^{-1}(\UU)\cap Z=\tot_{\UU}(\KK).$$

Hence, we have $\dim Z=\dim Z_W+\dim V-r$, or equivalently,
\begin{equation}\label{dim-count-eq}
\codim_W Z_W+r=\codim_{V\times W}Z=g'(P).
\end{equation}

\subsection{Proof of Theorem \ref{main-thm}} \hfill

\medskip

\noindent
{\bf Step 1. Choosing a general $\k$-point}. 
Shrinking the open subset $\UU\sub Z_W$ above, 
we can assume that $\UU$ is smooth. Since $\k$-points are dense in $Z_W$
we can choose a $\k$-point
$$w^0=(w^0_1,\ldots,w^0_{d-2})\in \UU\sub Z_W.$$

Let us set
$$S_V:=\ker(P_W(w^0):V\to U^*), \ \ S_U:=\ker(P_W(w^0)^*:U\to V^*).$$

\medskip

\noindent
{\bf Step 2. The first set of key tensors}. Set 
$$c:=\codim_W Z_W.$$
Since $w^0$ is a smooth point of $Z_W$, we can choose $c$ elements $g_1,\ldots,g_c$ in the ideal $I_{Z_W}\sub \k[W]$
with linearly independent derivatives at $w^0$. 
Now we recall that $W=W_1\times\ldots\times W_{d-2}$.
Thus, for each $a=1,\ldots,c$, and each nonempty subset of indices $I=\{i_1<\ldots<i_m\}\sub [1,d-2]$, we can
consider the polylinear forms, obtained as mixed derivatives at $w^0$,
$$g_{a,I}:=g_{a,w^0}^{(W_{i_1},\ldots,W_{i_m})}\in W_{i_1}^*\ot\ldots\ot W_{i_m}^*.$$

\medskip

\noindent
{\bf Step 3. Setting up the key identity.} 
Let us set $k=\dim V-r$.
Applying Lemma \ref{det-lem}(ii) to the morphism of trivial vector
bundles \eqref{P_W-morphism} over $Z_W$, we find global sections $v_1(w),\ldots,v_k(w)\in V\ot \k[Z_W]$, such that
$v_1(w^0),\ldots,v_k(w^0)$ form a basis of $S_V$, and
$$P(u,v_i(w),w)=0 \ \text{ for any } u\in U \ \text{ and } w\in Z_W, \ i=1,\ldots,k.$$ 
Since $\k[W]\to \k[Z_W]$ is surjective we can lift $v_i(w)$ to polynomials in $V\ot \k[W]$, which we denote in the same way.
Now we define a collection of $U^*$-valued polynomials on $W$,
\begin{equation}\label{key-identity}
f_i(w):=P(u,v_i(w),w)\in U^*\ot \k[W].
\end{equation}
By construction, all $f_i(w)$ belong to $U^*\ot I_{Z_W}\sub U^*\ot \k[W]$.
Equation \eqref{key-identity} will be the key identity that we will use.

\medskip

\noindent
{\bf Step 4. The second set of key tensors}.
We will consider certain mixed derivatives of $v_i(w)$, viewed as $V$-valued polynomials on $W$.
Namely, for each $I=\{i_1<\ldots<i_p\}\sub [1,d-2]$ we set
$$v_{i,I}:=v_{i,w^0}^{(W_{i_1},\ldots,W_{i_p})}\in W_I^*\ot V=\Hom(W_I,V),$$
where
$$W_I:=W_{i_1}\ot\ldots \ot W_{i_p}.$$
Since $(v_i(w^0))$ form a basis of $S_V$, there exists a unique operator
$$C_I:S_V\to \Hom(W_I,V): v_i(w^0)\mapsto v_{i,I}.$$
We extend $C_I$ in any way to an operator $V\to \Hom(W_I,V)$, which we still denote by $C_I$.
Note that we can also view $C_I$ as a linear map 
$$C_I:V\ot W_I\to V.$$
For an ordered collection of disjoint subsets $I_1,\ldots,I_p\sub [1,d-2]$, we consider the composition
$$C_{I_1}\ldots C_{I_p}:V\ot W_{I_1\sqcup\ldots \sqcup I_p}\xrightarrow{C_{I_p}} V\ot W_{I_1\sqcup\ldots I_{p-1}}\to\ldots 
\to V\ot W_{I_1} \xrightarrow{C_{I_1}} V.$$
We allow the case of an empty collection, i.e., $p=0$, in which case we just get the identity map $V\to V$.

Let us choose a basis $\ell_1,\ldots,\ell_r\in V^*$ in the orthogonal subspace to $S_V$.
For ordered collections $I_1\sqcup\ldots\sqcup I_p\sub [1,d-2]$ and for $j=1,\ldots,r$, we consider the polylinear forms
$$\ell_j\circ C_{I_1}\ldots C_{I_p}\in V^*\ot W_{I_1\sqcup\ldots\sqcup I_p}^*.$$
Note that for an empty collection, i.e., for $p=0$, we just get $\ell_j\in V^*$.

\medskip

\noindent
{\bf Step 5. Differentiating the key identity.}
For each $I=\{i_1<\ldots<i_p\}\sub [1,d-2]$, 
let us consider the embedding
$$\iota(I):W_I \to W_1\otimes\ldots\otimes W_{d-2},$$
which completes $w_{i_1}\ot\ldots\ot w_{i_p}$ by the components $w^0_j$ in the factors $W_j$ with $j\not\in I$.

Let us prove by induction on $p=0,\ldots,d-2$ that for any $I=\{i_1<\ldots<i_p\}\sub [1,d-2]$, one has
$$P|_{S_U\otimes V\otimes \iota(I)(W_I)}\in 
((\ell_j\circ C_{I_1}\ldots C_{I_s} \ |\  I_1\sqcup \ldots \sqcup I_s\subsetneq I,1\le j\le r,s\ge 0), (g_{a,I'} \ |\ 1\le a\le c, I'\sub I, I'\neq \emptyset)),
$$
where on the right we have the tensor ideal generated by the specified elements. Note that all the subsets $I_t$ are
supposed to be nonempty.

The base of induction $p=0$ is clear, since $P(u,v,w^0_1,\ldots,w^0_{d-2})=0$ for any $u\in S_U$ and $v\in V$.
Assume that $p>0$ and the assertion holds for $p-1$. Let us fix a subset
$I_0=\{i_1<\ldots<i_p\}\sub [1,d-2]$.

Now let us equate the $(W_{i_1},\ldots,W_{i_p})$-mixed derivatives at $w^0$ of both sides of the key identity \eqref{key-identity}.
We get the following equality in $U^*\ot W_{I_0}^*$:
\begin{equation}\label{differ-key-identity}
(f_i)^{(W_{i_1},\ldots,W_{i_p})}_{w^0}=
P|_{U\ot v_i(w^0)\ot \iota(I_0)W_{I_0}}+
\sum_{I\sqcup J=I_0, I\neq \emptyset}P|_{U\otimes C_Iv_i(w^0)\otimes \iota(J)W_J}.
\end{equation}
Note that by Lemma \ref{mixed-deriv-ideal-lem}, $(f_i)^{(W_1,\ldots,W_p)}_{w^0}$ belong to
the tensor ideal generated by $g_{a,I'}$ with $1\le a\le c$ and $I'\sub I_0$, $I'\neq \emptyset$.
Note also that the term in the sum in \eqref{differ-key-identity} corresponding to $J=\emptyset$
has zero restriction to $S_U$. Hence, we get
$$P|_{S_U\ot S_V\ot \iota(I_0)W_{I_0}}+\sum_{I\sqcup J=I_0; I,J\neq \emptyset}P|_{(\id_U\otimes C_I)(U\ot S_V)\otimes \iota(J)W_J}\in (g_{a,I'} \ |\ 1\le a\le c, I'\sub I, I'\neq \emptyset).$$
Now the induction assumption implies that $P|_{S_U\ot S_V\ot \iota(I_0)W_{I_0}}$ belongs to the tensor ideal generated by
$g_{a,I'}$ with $I'\sub I_0$, $I'\neq \emptyset$ and by the restrictions of
$\ell_j\circ C_{I_1}\ldots C_{I_s}$ with $s\ge 1$ (where $I_1\sqcup\ldots\sqcup I_s$ is a proper subset of $I_0$). 
By Lemma \ref{elem-lem}, adding $(\ell_j)$ to the generators of the tensor
ideal we get the required assertion about $P|_{S_U\ot V\ot \iota(I_0)W_{I_0}}$. 

\medskip

\noindent
{\bf Step 6. Conclusion of the proof for a single tensor}. 
Now using the result of the previous step for $p=d-2$, we get
$$\rk^S P|_{S_U\ot V\ot W_1\ot\ldots \ot W_{d-2}}\le r(1+\th_{d-2})+c(2^{d-2}-1),$$
where $\th_n$ is the number of ordered collections of disjoint nonempty
subsets $I_1\sqcup\ldots\sqcup I_p\subsetneq [1,n]$ (with $p\ge 1$).
By Lemma \ref{elem-lem}, this implies that
$$\rk^S P \le r+r(1+\th_{d-2})+c(2^{d-2}-1).$$
Now we recall that
$$r+c=g'(P)$$
(see \eqref{dim-count-eq}). Hence, we get
$$\rk^S P \le (r+c)\cdot\max(2+\th_{d-2},2^{d-2}-1)=g'(P)\cdot C_d$$
as claimed.

\medskip

\noindent
{\bf Step 7. The case of several tensors}.
Now assume that $\k$ is algebraically closed.
Suppose we are given a collection $\ov{P}=(P_1,\ldots,P_s)$ of polylinear forms on $V_1\times\ldots\times V_d$.
For a nonzero collection of coefficients $\ov{c}=(c_1,\ldots,c_s)$ in $\k$, we set 
$$P_{\ov{c}}=c_1P_1+\ldots+c_sP_s.$$
The key observation is that
\begin{equation}\label{Z-union-eq}
Z_{\ov{P}}=\bigcup_{\ov{c}\neq 0} Z_{P_{\ov{c}}}.
\end{equation}
where we can consider $\ov{c}$ as points in the projective space $\P^{s-1}$.
As we have already proved, for each $\ov{c}$,
$$\codim_{V_2\times\ldots\times V_d} Z_{P_{\ov{c}}}\ge C_d^{-1}\cdot \rk^S(P_{\ov{c}})\ge C_d^{-1}\cdot \rk^S(\ov{P}).$$
After taking the union over $\ov{c}$ in $\P^{s-1}$, we get
$$\codim_{V_2\times\ldots\times V_d} Z_{\ov{P}}\ge C_d^{-1}\cdot \rk^S(\ov{P})-s+1,$$
as claimed.

\section{Symmetric case}

\subsection{More on higher derivatives}

Let $f\in \k[V]_d$. 
Thinking of the $n$th derivative of $f\in \k[V]$ (where $n\le d$) as a degree $d-n$ polynomial map 
$$V\to \k[V]_n: v_0\mapsto f^{(n)}_{v_0}$$ 
we can write it as a tensor
$$f^{(n,d-n)}\in \k[V]_n\ot \k[V]_{d-n}.$$
By definition, 
$$f(v_1+v_2)=\sum_{n=0}^d f^{(n,d-n)}(v_1,v_2),$$
so $f^{(n,d-n)}$ is just the component of $f(v_1+v_2)$ of bidegree $(n,d-n)$ in $(v_1,v_2)$.

Similarly, we define an operation for $n_1+\ldots+n_p=d$,
$$\k[V]_d\to \k[V]_{n_1}\ot \ldots \ot \k[V]_{n_p}: f\mapsto f^{(n_1,\ldots,n_p)},$$
by letting $f^{(n_1,\ldots,n_p)}$ to be the component of multidegree $(n_1,\ldots,n_p)$ in
$f(v_1+\ldots+v_p)$.
For example,
$$f^{(1,1,d-2)}\in V^*\ot V^*\ot \k[V]_{d-2}$$
is exactly $H_f$, the Hessian symmetric form on $V$ (depending polynomially on $x\in V$).

We will use two properties of this construction, which are easy to check:
\begin{itemize}
\item $f^{(n_1,\ldots,n_p)}(x,\ldots,x)=\frac{d!}{n_1!\ldots n_p!}\cdot f(x)$;
\item for $m\le n_i$ the $m$th derivative with respect to $x_i$
of $f^{(n_1,\ldots,n_p)}(x_1,\ldots,x_p)$ at $(x^0_1,\ldots,x^0_p)$ is
equal to 
$$f^{(n_1,\ldots,n_{i-1},m,n_i-m,\ldots,n_p)}(x^0_1,\ldots,x^0_{i-1},v,x^0_i,\ldots,x^0_p).$$
\end{itemize}

\subsection{Proof of Theorem \ref{sym-main-thm}} \hfill

It will be convenient to denote one copy of $V$ as $X$ in the product $V\times V=V\times X$. 
In addition, we view $H_f=f^{(1,1,d-2)}$ as a bilinear form on $U\times V$ where $U=V$,
so that
$Z^{\sym}$ consists of pairs $(v,x)\in V\times X$ such that $f^{(1,1,d-2)}(u,v,x)=0$ for all $u\in U$.

\medskip

\noindent
{\bf Step 1. Dimension count and choosing a general $\k$-point}. 
Let $Z$ be an irreducible component of the Zariski closure of the set of $\k$-points $Z^{\sym}_f(\k)$,
such that $\codim_{V\times X}Z=g'_{\sym}(f)$, and let $Z_X\sub X$ denote the closure of the image of $Z$ under the
projection $p_2:V\times X\to X$. 
As before, we choose a nonempty smooth open subset $\UU\sub Z_X$ over which $H_f$
has maximal rank $r$, so that $p_2^{-1}(\UU)\cap Z$ is a vector bundle of rank $\dim V-r$ over $\UU$, 
in particular,
$$\codim_X Z_X+r=g'_{\sym}(f).$$ 

We choose a $\k$-point $x^0$ in $\UU\sub Z_X$ and set
$$S:=\ker(H_f(x^0))\sub V.$$

\medskip

\noindent
{\bf Step 2. The first set of key polynomials}. Set 
$$c:=\codim_X Z_X.$$
Since $x^0$ is a smooth point of $Z_X$, we can choose $c$ elements $g_1,\ldots,g_c$ in the ideal $I_{Z_X}\sub \k[X]$
with linearly independent derivatives at $x^0$. 
Thus, for each $a=1,\ldots,c$, and for $1\le i\le d-2$, we 
consider the derivatives
$$(g_a)^{(i)}_{x^0}\in \k[X]_i.$$

\medskip

\noindent
{\bf Step 3. Setting up key identity.} 
Let us set $k=\dim V-r$.
Applying Lemma \ref{det-lem}(ii) to the morphism of trivial vector
bundles $V\ot \OO\to V^*\ot \OO$ given by $H_f=f^{(1,1,d-2)}$
over $Z_X$, we find global sections $v_1(x),\ldots,v_k(x)\in V\ot \k[Z_X]$, such that
$v_1(x^0),\ldots,v_k(x^0)$ form a basis of $S$, and
$$f^{(1,1,d-2)}(u,v_i(x),x)=0 \ \text{ for any } u\in U \ \text{ and } x\in Z_X, \ i=1,\ldots,k.$$ 
We lift $v_i(x)$ to polynomials in $V\ot \k[X]$, which we denote in the same way.
Now we define a collection of $U^*$-valued polynomials on $X$,
\begin{equation}\label{sym-key-identity}
f_i(x):=f^{(1,1,d-2)}(u,v_i(x),x)\in U^*\ot \k[X].
\end{equation}
By construction, all $f_i(x)$ belong to $U^*\ot I_{Z_X}\sub U^*\ot \k[X]$.

\medskip

\noindent
{\bf Step 4. The second set of key forms}.
For each $1\le m\le d-2$,
we consider higher derivatives of $v_i$ at $x^0$, viewed as $V$-valued polynomials on $X$.
$$(v_i)^{(m)}_{x^0}\in V\ot \k[X]_m.$$
Since $(v_i(x^0))$ form a basis of $S$, there exists a linear operator
$$C_m:S\to V\ot \k[X]_m: v_i(x^0)\mapsto (v_i)^{(m)}.$$
We extend $C_m$ in any way to an operator $V\to V\ot \k[X]_m$, which we still denote by $C_m$.
For $m_1+\ldots+m_p\le d-2$, we consider the composition
$$C_{m_1}\ldots C_{m_p}:V\xrightarrow{C_{m_p}} V\ot \k[X]_{m_p}\to \ldots \to
V\ot \k[X]_{m_2+\ldots+m_p} \xrightarrow{C_{m_1}} V\ot \k[X]_{m_1+\ldots+m_p}.$$
We allow the case of an empty collection, i.e., $p=0$, in which case we just get the identity map $V\to V$.

Finally, we denote by $\ell_1,\ldots,\ell_r\in V^*$ a basis in the orthogonal subspace to $S$.
For $m_1+\ldots+m_p\le d-2$ and for $j=1,\ldots,r$, we consider the elements
$$\ell_j\circ C_{m_1}\ldots C_{m_p}\in V^*\ot \k[X]_{m_1+\ldots+m_p}.$$
Note that for an empty collection, i.e., for $p=0$, we just get $\ell_j\in V^*$.

\medskip

\noindent
{\bf Step 5. Differentiating the key identity.}
Let us prove by induction on $p=0,\ldots,d-2$ that one has
\begin{align*}
&f^{(1,1,p,d-2-p)}(u,v,x,x^0)|_{S\times V\times X}\in \\
&(((\ell_j\circ C_{m_1}\ldots C_{m_s})(v,x) \ |\ m_1+\ldots+m_s<p, 1\le j\le r),
((g_a)^{(m)}_{x^0}(x) \ |\ 1\le a\le c, 1\le m\le p),
\end{align*}
where on the right we have the ideal generated by the specified elements. 

The base of induction $p=0$ is clear, since $f^{(1,1,d-2)}(u,v,x^0)=0$ for any $u\in S$ and $v\in V$.
Assume that $p>0$ and the assertion holds for $p-1$.
Now let us equate the $p$th derivatives at $x=x^0$ of both sides of \eqref{sym-key-identity}.
We get the following equality in $U^*\ot \k[X]_p$:
$$(f_i)^{(p)}_{x^0}(x)=f^{(1,1,p,d-2-p)}(u,v_i(x^0),x,x^0)+\sum_{q=1}^p f^{(1,1,p-q,d-2-p+q)}(u,C_q(v_i(x^0),x),x,x^0).$$
The left-hand side belongs to the ideal generated by $(g_a)^{(m)}_{x^0}(x)$ with $1\le a\le c$ and $1\le m\le p$.
Note also that the term corresponding to $q=p$ in the right-hand side
has zero restriction to $u\in S$. Hence, we get
\begin{align*}
&f^{(1,1,p,d-2-p)}(u,v,x,x^0)|_{S\times S\times X}+\sum_{q=1}^{p-1} f^{(1,1,p-q,d-2-p+q)}(u,C_q(v,x),x,x^0)|_{S\times S\times X}\in 
\\ 
&((g_a)^{(m)}_{x^0}(x) \ |\ 1\le a\le c, 1\le m\le p).
\end{align*}
Now the induction assumption implies that 
$f^{(1,1,p,d-2-p)}(u,v,x,x^0)|_{S\times S\times X}$ belongs to the ideal generated by $(g_a)^{(m)}_{x^0}(x)$ for 
$1\le a\le c$, $1\le m\le p$ and by the restrictions to $S\times X$ of
$(\ell_j\circ C_{m_1}\ldots C_{m_s})(v,x)$ with $s\ge 1$, $m_1+\ldots+m_s<p$, $1\le j\le r$.
By Lemma \ref{elem-lem}, adding $(\ell_j)$ to the generators of the 
ideal we get the required assertion about $f^{(1,1,p,d-2-p)}(u,v,x,x^0)|_{S\times V\times X}$. 

\medskip

\noindent
{\bf Step 6. Conclusion of the proof for a single polynomial}. 
Now using the result of the previous step for $p=d-2$, we get
\begin{align*}
&f^{(1,1,d-2)}(u,v,x)|_{S\times V\times X}\in \\
&(((\ell_j\circ C_{m_1}\ldots C_{m_s})(v,x) \ |\ m_1+\ldots+m_s<d-2, 1\le j\le r),
((g_a)^{(m)}_{x^0}(x) \ |\ 1\le a\le c, 1\le m\le d-2)).
\end{align*}
Hence,
\begin{align}\label{arb-char-proof-eq}
&f^{(1,1,d-2)}(u,v,x)\in ((\ell_j(u) \ |\ 1\le j\le r), \\
&((\ell_j\circ C_{m_1}\ldots C_{m_s})(v,x) \ |\ m_1+\ldots+m_s<d-2, 1\le j\le r),
((g_a)^{(m)}_{x^0}(x) \ |\ 1\le a\le c, 1\le m\le d-2)).\nonumber
\end{align}
Now plugging $u=v=x$, we get
\begin{align*}
&d(d-1)\cdot f(x)\in \\
&((F_{j;m_1,\ldots,m_s}(x) \ |\ m_1+\ldots+m_s<d-2, 1\le j\le r),
((g_a)^{(m)}_{x^0}(x) \ |\ 1\le a\le c, 1\le m\le d-2)),
\end{align*}
where $F_{j;m_1,\ldots,m_s}(x)=(\ell_j\circ C_{m_1}\ldots C_{m_s})(x,x)$ has degree $1+m_1+\ldots+m_s<d-1$.
It follows that
$$\rk^S(f)\le r(1+\th^{\sym}_{d-2})+c(d-2),$$
where $\th^{\sym}_n$ is the number of $(m_1,\ldots,m_s)$, with $s\ge 1$, $m_i\ge 1$,
$m_1+\ldots+m_s<n$. It is easy to see that $\th^{\sym}_n=2^{n-1}-1$.
Since $r+c=g'_{\sym}(f)$, we get
$$\rk^S P \le (r+c)\cdot\max(2^{d-3},d-2)=g'_{\sym}(f)\cdot 2^{d-3}$$
as claimed.

\medskip

\noindent
{\bf Step 7. The case of several polynomials}.
Now assume that $\k$ is algebraically closed,
and we are given a collection $\ov{f}=(f_1,\ldots,f_s)$ of homogeneous polynomias on $V$ of degree $d$.
For a nonzero collection of coefficients $\ov{c}=(c_1,\ldots,c_s)$ in $\k$, we set 
$$f_{\ov{c}}=c_1f_1+\ldots+c_sf_s.$$
As in the non-symmetric case, the key observation is that
\begin{equation}\label{Zsym-union-eq}
Z^{\sym}_{\ov{f}}=\bigcup_{\ov{c}\neq 0} Z^{\sym}_{f_{\ov{c}}}.
\end{equation}
where we can consider $\ov{c}$ as points in the projective space $\P^{s-1}$.
Using the case of a single polynomial, we deduce that
$$\codim_{V\times V} Z^{\sym}_{\ov{f}}\ge 2^{-d+3}\cdot \rk^S(\ov{f})-s+1,$$
as claimed.

\subsection{Relation to singularities}\label{sing-sec}

Now we will relate $g_{\sym}(f)$ to $c(f)$, the codimension in $V$ of the singular locus of the hypersurface $f=0$.

\begin{prop}\label{cf-ineq-prop}
(i) The subvariety $Z^{\sym}_f\sub V\times X=V\times V$ contains the singular locus of $f^{(2,d-2)}(v,x)=0$.

\noindent
(ii) One has 
$$g_{\sym}(f)\le (d+1)\cdot c(f)$$ 
(resp., $g_{\sym}(f)\le d\cdot c(f)$ if $d$ is odd and $\mathrm{char}(\k)\neq 2$).

\noindent
(iii) If $\mathrm{char}(\k)$ does not divide $d-1$ then 
$$c(f)\le g_{\sym}(f).$$
\end{prop}

\begin{proof}
(i) The first derivative of $f^{(2,d-2)}(v,x)$ along $v$ at $(v^0,x^0)$ is
$f^{(1,1,d-2)}(v,v^0,x^0)$, so if $(v^0,x^0)$ is a singular point of $f^{(2,d-2)}(v,x)=0$ then
$f^{(1,1,d-2)}(v,v^0,x^0)=0$ for all $v$, i.e., $(v^0,x^0)\in Z^{\sym}_f$.

\noindent
(ii) Since we are comparing dimensions of algebraic varieties, without loss of generality, we can assume
that $\k$ is algebraically closed. 

By part (i), we have $g_{\sym}(f)\le c(F)$, where $F=f^{(2,g-2)}$.
It is easy to see that if $F(x)=F_1(x)+\ldots+F_r(x)$ then $c(F)\le c(F_1)+\ldots+c(F_r)$.
Also, if $A:V\to W$ is a linear surjective map and $g\in \k[W]$ then $c(g\circ A)=c(g)$.

Thus, it remains to check that $f^{(2,d-2)}(v,x)$ is a linear combination of $d+1$ (resp., $d$, if $d$ is odd
 and $\mathrm{char}(\k)\neq 2$)
polynomials of the form $f(A_i(v,x))$, for some linear surjective maps $A_i:V\times V\to V$.

Let us view $f(v+x)$ as a nonhomogeneous function of $v$, $g(v)=g_0+g_1+\ldots+g_d$ of degree $\le d$ 
(with coefficients in $\k[V]$). Now picking any $d+1$ distinct elements $\la_0,\ldots,\la_d\in\k$, we
can express $g_0,\ldots,g_d$ as linear combinations of $g(\la_0v),\ldots,g(\la_dv)$
(since the corresponding linear change is given by the Vandermonde matrix).  

In the case when $d$ is odd and $\mathrm{char}(\k)\neq 2$,
we can similarly express the components of even degree, $(g_{2i})_{i\le (d-1)/2}$
as linear combinations of $g_0=g(0)$ and $(g(\la_i v)+g(-\la_i v))/2$, for $1\le i\le (d-1)/2$,
where $(\la_i)$ are nonzero constants such that $(\la_i^2)$ are all distinct. 

It remains to observe that $g_2=f^{(2,d-2)}$ and that each $g(\la v)=f(\la v+x)$ is of the required type.

\noindent
(iii) This follows from the relation 
$$(d-1)\cdot f^{(1,d-1)}(v,x)=f^{(1,1,d-2)}(v,x,x).$$
Indeed, this implies that the intersection of $Z^{\sym}_f$ with the diagonal $V\sub V\times V$ is exactly the
singular locus of $f=0$, which gives the claimed inequality.
\end{proof}

Now let us consider the case of a collection $\ov{f}=(f_1,\ldots,f_s)$ of homogeneous polynomials on $V$ of degree $d$.
We consider the corresponding family of hypersurfaces in $V$, $f_{\ov{c}}=0$ parametrized by the projective space $\P^{s-1}$.
It is clear that for the locus $S(\ov{f})\sub V$ where the rank of Jacobi matrix of $(f_1,\ldots,f_s)$ is $<s$, we have
$$S(\ov{f})=\bigcup_{\ov{c}\neq 0} \Sing(f_{\ov{c}}=0).$$

\begin{prop}\label{cf-ineq-bis-prop}
(i) One has the inclusion
$$\bigcup_{\ov{c}\neq 0} \Sing(f_{\ov{c}}^{(2,d-2)}=0)\sub Z^{\sym}_{\ov{f}}.$$

\noindent
(ii) One has
$$g_{\sym}(\ov{f})\le (d+1)\cdot c'(\ov{f})+d(s-1)$$
(resp., $g_{\sym}(\ov{f})\le d\cdot c'(\ov{f})+(d-1)(s-1)$ if $d$ is odd and $\mathrm{char}(\k)\neq 2$).

\noindent
(iii) Assume that $(f_1,\ldots,f_s)$ define a complete intersection
$V(\ov{f})\sub V$, i.e., $\codim_V V(\ov{f})=s$.
Then 
$$c'(\ov{f})\le c(\ov{f})\le c'(\ov{f})+s.$$
Assume in addition that $\mathrm{char}(\k)$ does not divide $d-1$. Then
$$c'(\ov{f})\le g_{\sym}(\ov{f}).$$
\end{prop}

\begin{proof}
(i) This follows from Proposition \ref{cf-ineq-prop}(i) due to \eqref{Zsym-union-eq}.

\noindent
(ii) Since $S(\ov{f})$ has codimension $c'(\ov{f})$ in $V$,
it follows that for some $a\le s-1$, there exists an $a$-dimensional subvariety $X\sub \P^{s-1}$
such that 
$$c(f_{\ov{c}})=\codim_V \Sing (f_{\ov{c}})\le c'(\ov{f})+a \ \text{ for } \ov{c}\in X.$$
Applying Proposition \ref{cf-ineq-prop}(ii) we see that for each $\ov{c}\in X$, one has
$$\codim_{V\times V} Z^{\sym}_{f_{\ov{c}}}\le (d+1)\cdot (c'(\ov{f})+a)$$
(resp., $\le d\cdot (c'(\ov{f})+a)$ if $d$ is odd).
Hence, using \eqref{Zsym-union-eq} we get
$$\codim_{V\times V} Z^{\sym}_{\ov{f}}\le (d+1)\cdot (c'(\ov{f})+a)-a$$
(resp., $\le d\cdot (c'(\ov{f})+a)-a$ if $d$ is odd). Since $a\le s-1$, this implies the assertion.

\noindent
(iii) If $(f_1,\ldots,f_s)$ define a complete intersection then by the Jacobi criterion of smoothness, we have
$$\Sing V(\ov{f})=S(\ov{f})\cap V(\ov{f}).$$
In particular, we have an inclusion $\Sing V(\ov{f})\sub S(\ov{f})$, so 
$$c'(\ov{f})=\codim_V S(\ov{f})\le c(\ov{f}).$$
Also, we have
$$c(\ov{f})-s=\codim_{V(\ov{f})} \Sing V(\ov{f})\le \codim_V S(\ov{f})=c'(\ov{f}).$$

If we assume in addition that $\mathrm{char}(\k)$ does not divide $d-1$ then 
the intersection of $Z^{\sym}_{\ov{f}}$ with the diagonal $V\sub V\times V$ is exactly $S(\ov{f})$.
Hence, we get 
$$c'(\ov{f})=\codim_V S(\ov{f})\le g_{\sym}(\ov{f}).$$
\end{proof}

\medskip
\begin{proof}[Proof of Theorem \ref{cf-thm}]
(i) If $f(x)=\sum_{i=1}^r h_i(x)g_i(x)$ then the locus $h_i(x)=g_i(x)=0$, for $i=1,\ldots,r$, is contained in the singular locus of
$f(x)=0$, so $c(f)\le 2r$.

 Now for the other inequality, let $c =c_\k (f) $ and $X$ an irreducible component of codimension $c$ of the Zariski closure of the $\k$-points of $\Sing(f=0).$ Let $v_0\in X$ be a smooth $\k$-point and $g_1,\ldots,g_c\in I(X)$ defined over $\k$ with linearly independent differentials at $v_0.$ For all $k\in [n], \partial_k f \in I(X),$ so lemma \ref{deriv-ideal-lem} yields
    \[
        \partial_k f = (\partial_k f)^{(d-1)}_{v_0} \in \left((g_i)^{(j)}_{v_0}\right)_{i\in [c],j\in [d-1]}.
    \]
    By Euler's formula,
    \[
        f = \frac{1}{d}\sum_{k=1}^n x_k \partial_k f\in \left((g_i)^{(j)}_{v_0}\right)_{i\in [c],j\in [d-1]}.
    \] 
    This gives $\rk^S(f)\le (d-1)\cdot c.$

\noindent
(ii) We deduce this from the result for a single form as in the proof of theorem \ref{sym-main-thm}.
%
%
\end{proof}

\medskip
\begin{proof}[Proof of Corollary \ref{AH-cor}]
In the notation of \cite[Theorem A]{AH} (recalling that the strength of $f$ is $\rk^S(f)-1$), the inequality of Theorem \ref{cf-thm}(i) implies that one can take 
$$\prescript{m}{}A(d)=(d-1)(m+2)-1.$$
It is also well known that for $d=2$, one can take 
$$\prescript{m}{}A(2)=\lceil \frac{m+1}{2}\rceil$$
(see e.g., \cite[Prop.\ 4.10]{AH2}).
Now the assertion follows from \cite[Theorem A(c)]{AH}.
\end{proof}

\begin{remark}\label{char-free-rem}
Note that for $\k$ algebraically closed of {\it arbitrary} characteristic, Eq.\,\eqref{arb-char-proof-eq} shows that 
$$\rk^S f^{(1,1,d-2)}(u,v,x) \le (2^{d-3}+1)\cdot g_{\sym}(f)\le (2^{d-3}+1)\cdot (d+1)\cdot c(f).$$
\end{remark}

The proof of Theorem \ref{derivative-thm} is based on the following geometric observation.

\begin{lemma}\label{deriv-codim-lem}
For generic $v_1,\ldots,v_s\in V$, where $s<\dim V$, we have
$$c'(\partial_{v_1}f,\ldots,\partial_{v_s}f)\ge g_{\sym}(f)-s+1.$$
\end{lemma}

\begin{proof}
Let us denote by $Z^{(s)}\sub V^s\times V$ the locally closed subvariety consisting of $(v_1,\ldots,v_s,x)$, such that 
$v_1,\ldots,v_s$ are linearly independent and $\dim\sspan (H_f(x)(\cdot,v_1),\ldots,H_f(x)(\cdot,v_s))<s$
(here we consider $H_f(\cdot,v)$ as a linear form on $V$). We want to estimate the dimension of $Z^{(s)}$.
We have a surjective map (with at least $1$-dimensional fibers)
$$\wt{Z}^{(s)}\to Z^{(s)},$$
where $\wt{Z}^{(s)}\sub V\times V^s\times V$ is given by
$$\wt{Z}^{(s)}=\{(v,v_1,\ldots,v_s,x) \ |\ v\in \ker H_f(x), v\neq 0, (v_1,\ldots,v_s) \text{ linearly independent}, v\in \sspan(v_1,\ldots,v_s)\}.$$
We have a natural projection 
$$\wt{Z}^{(s)}\to Z_f^{\sym}: (v,v_1,\ldots,v_s,x)\mapsto (v,x),$$
which is a locally trivial fibration whose fibers are irreducible of dimension $n(s-1)+s$, where $n=\dim V$.
It follows that 
$$\dim Z^{(s)}\le \dim \wt{Z}^{(s)}-1\le \dim Z_f^{\sym}+n(s-1)+s-1.$$
Hence,
$$\codim_{V^s\times V} Z^{(s)}\ge g_{\sym}(f)-s+1.$$

Next, we observe that $H_f(x)(\cdot,v)=f^{(1,1,d-2)}(\cdot,v,x)=(\partial_v f)^{(1,d-2)}(\cdot,x)$, so
$S(\partial_{v_1}f,\ldots,\partial_{v_s}f)$ is exactly the fiber over $(v_1,\ldots,v_s)$ of the projection $Z^{(s)}\to V^s$.
For generic $v_1,\ldots,v_s$, only the components of $Z^{(s)}$ dominant over $V^s$ will play a role, and we deduce that
$$c'(\partial_{v_1}f,\ldots,\partial_{v_s}f)=\codim_V S(\partial_{v_1}f,\ldots,\partial_{v_s}f)\ge \codim_{V^s\times V} Z^{(s)}\ge g_{\sym}(f)-s+1.$$
\end{proof}

\medskip

\begin{proof}[Proof of Theorem \ref{derivative-thm}]
(i) By Lemma \ref{deriv-codim-lem} with $s=1$, $c(\partial_v f)\ge g_{\sym}(f)$. Hence, by Theorems \ref{cf-thm}(i) and \ref{sym-main-thm},  
$$\rk^S(\partial_v f)\ge \frac{1}{2}\cdot c(\partial_v f)\ge \frac{1}{2}\cdot g_{\sym}(f)\ge 2^{2-d}\cdot \rk^S(f).$$

\noindent
(ii) If $c'(\partial_{v_1}f,\ldots,\partial_{v_s}f)\ge s$
(resp., $c'(\partial_{v_1}f,\ldots,\partial_{v_s}f)\ge s+2$) then $(\partial_{v_1}f,\ldots,\partial_{v_s}f)$ define a (resp., normal) complete intersection of
codimension $s$. Hence, the assertion follows from Theorem \ref{sym-main-thm} and Lemma \ref{deriv-codim-lem}.
\end{proof}

\subsection{Singularities of the polar map}\label{sing-polar-sec}

Let $f\in \k[V]_d$.
Note that $H_f(x)$ can be identified with the tangent map to the {\it polar map} $\phi_f:V\to V^*$ of $f$ sending $x$ to
$f^{(1)}_x=df|_x$. Thus, $g_{\sym}(f)$ measures the degeneracy of this map.

More precisely, for any morphism $\phi:X\to Y$ between smooth connected varieties, let us define the 
{\it Thom-Boardman rank}\footnote{The name is due to the relation with Thom-Boardman stratification in singularity theory, see \cite{Board}.} 
of $\phi$, denoted as $\rk^{TB}(\phi)$, as follows. 
Consider the subvariety $Z_\phi$ in the tangent bundle $TX$ of $X$ consisting
of $(x,v)$ such that $d\phi_x(v)=0$. Then we set
$$\rk^{TB}(\phi)=\codim_{TX} Z_\phi.$$
Note that $\rk^{TB}(\phi)\le r$, where $r$ is the generic rank of the differential of $\phi$, however, the inequality can be strict.

By definition, 
$$g_{\sym}(f)=\rk^{TB}(\phi_f).$$ 
As is well known, the generic rank of $d\phi_f=H_f$ is related to the dimension of
the projective dual variety $X^*$ of the projective hypersurface associated with $f$ (more precisely, $\dim X^*+2$
is the generic rank of $H_f$ over the hypersurface $f=0$). However, it is easy to see that $g^{\sym}(f)$ can be
much smaller than the generic rank of $\phi_f$.
For example, if $q_1(x)$ and $q_2(y)$ are nondegenerate quadratic forms in two different groups of variables
$(x_1,\ldots,x_n)$, $(y_1,\ldots,y_n)$,
then $\rk^S(q_1(x)q_2(y))=1$, so $g_{\sym}(q_1(x)q_2(y))\le 4$.
On the other hand, the generic rank of $\phi_{q_1(x)q_2(y)}$ is $2n$ (assuming the characteristic of $\k$ is $\neq 2,3$).

\begin{example}
In the case $d=3$, the Schmidt rank of $f$ is equal to its slice rank $s(f)$, i.e., the minimal $s$ such that there exists a
linear subspace $L\sub V$ of codimension $s$ contained in $(f=0)$. 
Thus, for a cubic form $f$, assuming that $\k$ is algebraically closed of characteristic $\neq 2,3$, we get from
Theorem \ref{sym-main-thm} and from \eqref{sym-trivial-ineq} that
$$s(f)\le \rk^{TB}(\phi_f)\le 4 s(f).$$
If $f$ is a general homogeneous polynomial of degree $d$ then we still have $\rk^S(f)=s(f)$ (see \cite{BBOV1}), so for such $f$,
assuming $\k$ to be algebraically closed of characteristic not dividing $(d-1)d$, we have
$$2^{3-d}s(f)\le \rk^{TB}(\phi_f)\le 4 s(f).$$
\end{example}

It seems that the invariant $\rk^{TB}(\phi)$ deserves to be studied more. For example, we do not know whether
it is always true that $\rk^{TB}(\phi)=\dim X$ for a {\it finite} morphism $\phi$ between smooth projective varieties
in characteristic zero.
Note the following corollary from Prop.\ \ref{cf-ineq-prop}(iii).

\begin{cor}\label{TB-rank-cor} 
Assume that $\mathrm{char}(\k)$ does not divide $d-1$. Then
$$\rk^{TB}(\phi_f)\ge c(f).$$
In particular, if the projective hypersurface associated with $f$ is smooth then
$$\rk^{TB}(\phi_f)=\dim V.$$
\end{cor}

Let $V_f\sub \P V$ denote the projective hypersurface associated with $f$.
In \cite{CD} the authors consider (for $\k=\C$) the closed locus $S_{\ge r}\sub V_f$
where the co-rank of the Hessian $H_f$ is $\ge r$. They prove that if $V_f$ is smooth then
for $r(r+1)\le \dim V$, the subvariety $S_{\ge r}(V)$ is nonempty and
$$\codim_{V_f} S_{\ge r}(V) \le r(r+1)/2.$$
Using Corollary \ref{TB-rank-cor} we get the inequality
$$\codim_{V_f} S_{\ge r}(V) \ge r-1.$$

If $V_f$ is smooth then the projectivization of the 
restriction of $\phi_f$ to $(f=0)$ can be identified with the {\it Gauss map} 
$$\ga:V_f\to \P V^*.$$
It is easy to check that if $\mathrm{char}(k)$ does not divide $d(d-1)$ then
for any point $x\in (f=0)\sub V$ one has $\ker(d(\phi_f)_x)\sub T_x(f=0)$ and the natural projection
$$\ker(d(\phi_f)_x)\to \ker(d\ga_x)$$
is an isomorphism. Thus, the above inequalities can be viewed as 
restrictions on possible degeneracies of the Gauss map of $V_f$ (which is finite by
a result of Zak in \cite{Zak}).

\appendix
\section{}

This appendix gives alternative versions of theorems \ref{main-thm}, \ref{sym-main-thm}, and \ref{derivative-thm}, with better bounds for $d \ge 6$. Here is the the second version of theorem \ref{main-thm}:

\begin{theorem}\label{tensor}
(i) Let $g'(P)$ denote the codimension in $V_2\times\ldots\times V_d$ of the Zariski closure of $Z_P(\k).$
Then one has
$$\rk^S(P)\le (2^{d-1}-1)\cdot g'(P).$$

\noindent
(ii) Assume $\k$ is algebraically closed. Then
for a collection $\ov{P}=(P_1,\ldots,P_s)$, one has
$$\rk^S(\ov{P})\le (2^{d-1}-1)\cdot (g(\ov{P})+s-1).$$
\end{theorem}

For algebraically closed fields the above result matches the one obtained by Cohen and Moshkowitz \cite{CM}, but we give a very short proof.

\begin{proof}
    
     (i) The proof will mimic that of theorem \ref{cf-thm}. Write $g=g'(P)$ and let $X$ be an irreducible component of the Zariski closure of 
    $Z_P(\k)$ such that $\codim_{V_1\times V_2\times \ldots \times V_{d-1}} X=g.$  Let $x_1,\ldots,x_n$ be a basis for $V_d^*.$ Write $P = \sum_{k=1}^n x_k\cdot Q_k,$ where $Q_k:V_1\times V_2\times \ldots \times V_{d-1}\to \k $ are polylinear forms.  Let $v_0\in X$ be a smooth $\k$-point and $h_1,\ldots,h_g\in I(X)$ defined over $\k$ with linearly independent differentials at $v_0.$ For all $k\in [n], Q_k = (Q_k)_{v_0}^{(V_1,V_2,\ldots,V_{d-1})} \in I(X),$ so by lemma \ref{mixed-deriv-ideal-lem} it is in the tensor ideal generated by
    \[
        \left((h_i)^{{(V_{j_1},\ldots,V_{j_s})}}_{v_0}\right)_{i\in [g],\emptyset\neq\{j_1<\ldots<j_s\}\subset [d-1]}.
    \]
    By definition, $P$ is in the tensor ideal generated by the $Q_k,$ so $\rk^S(P)\le (2^{d-1}-1)\cdot g.$
    
    \noindent (ii) We deduce this from the result for a single tensor as in the proof of theorem \ref{main-thm}.
\end{proof}

The second version of theorem \ref{sym-main-thm} is:
\begin{theorem}\label{sym-alt}
    Assume that $\k$ is algebraically closed of characteristic not dividing $(d-1)d$. 
    (i)  For a single form $f$ of degree $d,$
    $$\rk^S(f)\le (d-1)\cdot g_{\sym}(f).$$ 
    \noindent
    (ii) If $(f_1,\ldots,f_s)$ define a complete intersection of codimension $s$ in $V$, then
    $$\rk^S(\ov{f})\le (d-1)\cdot (g_{\sym}(\ov{f})+s-1).$$
\end{theorem}

\begin{proof}
    (i) Combine theorem \ref{cf-thm}(i) and proposition \ref{cf-ineq-prop}(iii).
    
    \noindent (ii) Combine theorem \ref{cf-thm}(ii) and proposition \ref{cf-ineq-bis-prop}(iii).
\end{proof}

And the second version of theorem \ref{derivative-thm}:
\begin{theorem}
    Let $f$ be a homogeneous polynomial of degree $d$.
Assume that $\k$ is algebraically closed of characteristic not dividing $(d-1)d$.

\noindent
(i) For generic $v\in V$, one has $\rk^S(\partial_v f)\ge \frac{1}{2d-2}\cdot \rk^S(f)$.

\noindent
(ii) For $s\le \frac{1}{2d-2}\cdot \rk^S(f)+\frac{1}{2}$ (resp., $s\le \frac{1}{2d-2}\cdot \rk^S(f)-\frac{1}{2}$), and for generic $v_1,\ldots,v_s\in V$, the derivatives $(\partial_{v_1}f,\ldots,\partial_{v_s}f)$ define a (resp., normal)
complete intersection of codimension $s$ in $V$.
\end{theorem}

\begin{proof}
    (i) By Lemma \ref{deriv-codim-lem} with $s=1$, $c(\partial_v f)\ge g_{\sym}(f)$. Hence, by Theorems \ref{cf-thm}(i) and \ref{sym-alt},  
$$\rk^S(\partial_v f)\ge \frac{1}{2}\cdot c(\partial_v f)\ge \frac{1}{2}\cdot g_{\sym}(f)\ge \frac{1}{2d-2}\cdot \rk^S(f).$$

\noindent
(ii) If $c'(\partial_{v_1}f,\ldots,\partial_{v_s}f)\ge s$
(resp., $c'(\partial_{v_1}f,\ldots,\partial_{v_s}f)\ge s+2$) then $(\partial_{v_1}f,\ldots,\partial_{v_s}f)$ define a (resp., normal) complete intersection of
codimension $s$. Hence, the assertion follows from Theorem \ref{sym-alt} and Lemma \ref{deriv-codim-lem}.
\end{proof}

\end{document}